\documentclass[11pt]{amsart}
\usepackage{amsfonts}
\usepackage{mathrsfs}
\usepackage{bbm}
\usepackage{amsmath}
\usepackage{amssymb}
\usepackage{amscd}
\usepackage{amsthm}
\usepackage[compress]{cite}
\usepackage{enumerate}
\usepackage{hyperref}
\hypersetup{
colorlinks=true,
linkcolor=blue,
anchorcolor=black,
citecolor=blue
}

\makeatletter
\@namedef{subjclassname@2020}{\textup{2020} Mathematics Subject Classification}
\makeatother

\newtheorem{theorem}{\indent Theorem}[section]
\newtheorem{lemma}[theorem]{\indent Lemma}
\newtheorem{proposition}[theorem]{\indent Proposition}
\newtheorem{corollary}[theorem]{\indent Corollary}
\newtheorem{definition}[theorem]{\indent Definition}
\newtheorem{example}[theorem]{\indent Example}
\newtheorem{remark}[theorem]{\indent Remark}

\numberwithin{equation}{section}

\begin{document}

\title{A family of parametric isoperimetric-type inequalities with multiple geometric quantities}

\author{Heran Zhao}

\address{Beijing International Center for Mathematical Research \& School of Mathematical Sciences \\ Peking University \\Beijing, 100871, P. R. China}

\email{zhaohr618@stu.pku.edu.cn}



\thanks{The author would especially like to appreciate professor Xiang Gao for meaningful discussions.}

\begin{abstract}
We establish a family of parametric isoperimetric-type inequalities with multiple
geometric quantities for closed convex curves. 
These inequalities hold under certain parameter conditions. 
We also prove the equality conditions. 
Some new inequalities and improved versions of proven inequalities are derived by adjusting the parameters. 
Such inequalities are stable under stronger conditions.
\end{abstract}


\maketitle

\pagestyle{myheadings} \markboth{\centerline{Heran Zhao}}{\centerline{Heran Zhao}}


\section{Introduction and main results}

The classical isoperimetric inequality states that for a closed curve $\gamma \subset \mathbb{R}^2$ with length $L$, enclosing area $A$, then
\begin{equation}\label{1.1}
L^2-4\pi A \geq 0
\end{equation}
with equality if and only if $\gamma$ is a circle. 
Although known since ancient Greeks, the first strict proof was given by Steiner \cite{S}. 
In the 1920s, a class of inequalities of the form
\begin{equation}\label{1.2}
L^2-4\pi A \geq B,
\end{equation}
was introduced by Bonnesen, where $B$ is sum of non-negative geometric quantities and vanishes only for circles. 
Hence \eqref{1.2} is an improved form of \eqref{1.1}, called Bonnesen-type inequality.
A typical example due to Bonnesen \cite{B} is
\begin{equation}\label{1.3}
L^2-4\pi A \geq \pi^2\left(\rho_e - \rho_i\right)^2,
\end{equation}
where $\rho_e$ and $\rho_i$ denote the circumradius and inradius of $\gamma$, respectively, with equality if and only if $\gamma$ is a circle.
Then many Bonnesen-type inequalities are proved and generalized (See \cite{O, PS, Sc, XZZ1, ZZY, Z, ZZ, ZC}). 

There are also some reverse Bonnesen-type inequalities of the form
$$
L^2-4\pi A \leq B.
$$
Bottema \cite{Bo} proved the inequality
\begin{equation}\label{1.4}
    L^2 - 4\pi A \leq \pi^2(\rho_M - \rho_m)^2,
\end{equation}
where $\rho_M$ and $\rho_m$ are maximum and minimum of the curvature radius $\rho$ of $\gamma$, respectively, with equality if and only if $\gamma$ is a circle.
Pan and Zhang \cite[Theorem 4.2]{PZ} proved
\begin{equation}\label{1.5}
    L^2 \leq 4\pi \left(A +  \left|\tilde A\right|\right),
\end{equation}
where $\tilde A$ is the oriented area of the domain enclosed by the curvature centers locus $\beta$ of $\gamma$, with equality if and only if $\gamma$ is a circle.
The inequality 
\begin{equation}\label{1.6}
\pi\int_0^{2\pi}\rho(\theta)^2\mathrm{d}\theta 
\geq L^2 - 2\pi A
\end{equation}
was proved by Pan and Yang in \cite[Theorem 2.1]{PY}, where $\theta$ is the angle between $x$-axis and the
outward normal vector. 
Gao proved 
\begin{equation}\label{1.7}
    L^2 \leq 4\pi A + \pi \left|\tilde A\right|
\end{equation}
and 
\begin{equation}\label{1.8}
    \pi\int_0^{2\pi}\rho\left(\theta\right)^2\mathrm{d}\theta\geq L^2-2\pi A+\pi\left|\tilde{A}\right|
\end{equation}
in \cite[Corollary 1.3]{G2}, which give improved versions of \eqref{1.5} and \eqref{1.6}, respectively.
More reverse Bonnesen-type inequalities can be found in \cite{BWZ, G1, LG, W1, W2, XC, XZZ, ZD, ZMZC, ZPMW, ZZ, ZX}.

In the 1980s, the curvature $\kappa$ of $\gamma$ was introduced to strengthen and generalize classical isoperimetric inequalities. Gage proved 
$$
\oint_{\gamma}\kappa^2\mathrm{d}s
\geq
\frac{\pi L}{A}
$$
in \cite{Ga1}, which is called Gage inequality. 
More Gage-type inequalities were given in \cite{MZ, Ga2, Ga3, LG}.
Curvature flow methods are widely used to prove inequalities involving $\kappa$ (See \cite{HY, GPT1, GPT2, K, CWW}).
These isoperimetric-type inequalities are applied in mathematical physics and mechanics. 

Stability estimates for isoperimetric-type inequalities are studied as well. Pan and Xu \cite[Theorem 3.1 and 3.2]{PX} obtained 
$$
h_{1}\left(K,S\left(K\right)\right)^{2}\leq\frac{4\pi^2-33}{96\pi^2}\left(4\pi\left(A\left(K\right)+\left|\tilde{A}\left(K\right)\right|\right)-L^2\left(K\right)\right),
$$
$$
h_{2}\left(K,S\left(K\right)\right)^{2}\leq\frac{1}{18\pi}\left(4\pi\left(A\left(K\right)+\left|\tilde{A}\left(K\right)\right|\right)-L^{2}\left(K\right)\right)
$$
for \eqref{1.5} by comparing a convex domain $K$ with its Steiner disk $S(K)$ (See Definition \ref{D1}), where $h_1$ and $h_2$ are $L^{\infty}$ and $L^2$ norms respectively (See Definition \ref{D}). 
Stability results for other inequalities can be found in \cite{G1, G2, GSL, KNPS, LG}.

Above results cover only a limited set of geometric quantities. 
The optimum stability condition remains unaddressed. 
The main results of this paper is that we extend previous results by incorporating a broader range of geometric quantities and giving the associated stability conditions, i.e., we have the following theorems.

\begin{theorem}\label{T1} 
    Let $K \subset \mathbb{R}^2$ be a closed and strictly convex domain enclosed by smooth curve $\gamma$. 
    Let $\beta$ denote the curvature centers locus of $\gamma$ and $\rho_{\beta}(\theta)$ denote curvature radius of $\beta$. 
    Geometric quantities $\rho$, $L$, $A$, $\tilde A$, $\rho_e$, $\rho_i$, $\rho_M$, $\rho_m$ and $\kappa$ defined as above. 
    \begin{enumerate}[(i)]
        \item \label{1.1.1}
        For parameter set $(\alpha, \delta, \mu, \sigma, \eta, \lambda, \xi, \zeta)$ satisfying
        \begin{equation}\label{1.9}
            \begin{cases}
                \begin{array}{cl}
            \eta, \xi, \zeta \geq 0,\\
            \lambda \leq 0,\\
            2\alpha + 4\pi\delta + \mu - 4\zeta \geq 0,\\
            2\alpha + \sigma - 2\eta - 4\zeta \geq 0,\\
            6\pi\alpha - \pi\mu + 4\pi\sigma + 24\pi\eta + 4\lambda + 4\xi - 12\pi\zeta \geq   0,\\
                \end{array}
            \end{cases}
        \end{equation}
        we have
        \begin{equation}\label{1.10}
        \begin{split}
        W:= & \alpha \int_0^{2\pi}\rho(\theta)^2\mathrm{d}\theta
        + \delta L^2 +\mu A + \sigma \left|\tilde A\right| + \eta \int_0^{2\pi}\rho_{\beta}(\theta)^2\mathrm{d}\theta \\
        & + \lambda(\rho_e - \rho_i)^2 + \xi (\rho_{M}-\rho_{m})^2 + \zeta \oint_{\gamma}\kappa^2\mathrm{d}s
        \geq 0.  
        \end{split}
        \end{equation}  
        \item \label{1.1.2}
        The equality holds in \eqref{1.10} if $K$ is a disk and the parameter set satisfies 
    \begin{equation}\label{1.11}
        \begin{cases}
            \begin{array}{cl}
                2\alpha+4\pi \delta+\mu=0,\\
                \zeta = 0.\\
            \end{array}
        \end{cases}
    \end{equation}
        \item \label{1.1.3}
        Conversely, if the equality holds in \eqref{1.10} and parameter set satisfies \eqref{1.9} and 
        \begin{equation}\label{1.12}
            \lambda^2 + \xi^2 + \zeta^2 > 0,
        \end{equation}
        then $K$ is a disk.
    \end{enumerate}
\end{theorem}

The stability of inequality \eqref{1.10} is proved in Theorem \ref{T2} and \ref{T3} with respect to $h_1$ and $h_2$ norms (See Definition \ref{D}).

\begin{theorem}\label{T2}
Let $S(K)$ denote the Steiner disk associated with $K$, and $p_K(\theta)$, $p_{S(K)}(\theta)$ denote the support function of $K$, $S(K)$, respectively. 
\begin{enumerate}[(i)]
    \item\label{1.2.1} For parameter set satisfying
        \begin{equation}\label{1.13}
            \begin{cases}
                \begin{array}{cl}
            \eta, \xi, \zeta \geq 0,\\
            \lambda \leq 0,\\
            2\alpha + 4\pi\delta + \mu - 4\zeta \geq 0,\\
            2\alpha + \sigma - 2\eta - 4\zeta \geq 0,\\
            6\pi\alpha - \pi\mu + 4\pi\sigma + 24\pi\eta + 4\lambda + 4\xi - 12\pi\zeta > 0,\\
                \end{array}
            \end{cases}
        \end{equation}
    we have 
    \begin{equation}\label{1.14}
    \left(
    \max_{\theta \in [0, 2\pi]} \left|p_K(\theta)-p_{S(K)}(\theta)\right|
    \right)^2
    \leq 
    C(\alpha, \mu, \sigma, \eta, \lambda, \xi, \zeta) W,
    \end{equation}
    where
    \[  C(\alpha, \mu, \sigma, \eta, \lambda, \xi, \zeta)=\max \left\{
        1,
        \frac{3}{2(6\pi\alpha - \pi\mu + 4\pi\sigma + 24\pi\eta + 4\lambda + 4\xi - 12\pi\zeta)}
    \right\}. 
    \]
    \item The equality holds in \eqref{1.14} if $K$ is a disk and the parameter set satisfies \eqref{1.11}.
\end{enumerate}
\end{theorem}

\begin{theorem}\label{T3}
Under the assumptions of Theorem \ref{T2}, we have the following.
\begin{enumerate}[(i)]
    \item\label{1.3.1} 
    For parameter set satisfying \eqref{1.13},
we have
\begin{equation}\label{1.15}
\int_0^{2\pi}\left|p_K(\theta)-p_{S(K)}(\theta)\right|^2\mathrm{d}\theta
\leq 
C(\alpha, \mu, \sigma, \eta, \lambda, \xi, \zeta) W,  
\end{equation}
where
$$
C(\alpha, \mu, \sigma, \eta, \lambda, \xi, \zeta) 
= \frac{2\pi}{3(6\pi\alpha - \pi\mu + 4\pi\sigma + 24\pi\eta + 4\lambda + 4\xi - 12\pi\zeta)}.
$$
    \item
    The equality holds in \eqref{1.15} if $K$ is a disk and the parameter set satisfies \eqref{1.11}. 
    \item 
    Conversely, if the equality holds in \eqref{1.15} and parameter set satisfies \eqref{1.12} and \eqref{1.13}, then $K$ is a disk.
        
\end{enumerate}
\end{theorem}

The rest of this paper is organized as follows. Some geometric quantities and their Fourier expansions are given in Section \ref{S2}. 
Combining the results from Section \ref{S2}, Theorem \ref{T1} is proved in Section \ref{S3}. 
In Section \ref{S4}, we prove the stability of inequality \eqref{1.10} with respect to $h_1$ and $h_2$ norms in Theorem \ref{T2} and \ref{T3}. 
Then we identify the equality conditions in Theorem \ref{T1}, \ref{T2} and \ref{T3}. 
Furthermore, we prove the stability conditions are stronger than \eqref{1.9} (See Remark \ref{R}).

\section{Preliminaries} \label{S2}

This section introduces the Fourier expansions of geometric quantities related to Theorem \ref{T1}. 
We assume strict convexity of $\gamma$ to ensure $\kappa(\theta)>0$ and Fourier series needed in the proof convergent uniformly. 
More details can be found in \cite[Chapter 4]{Gr}.
Let $p_K(\theta)$ denote the support function of domain $K$, where $\theta$ is the angle between the $x$-axis and the outward normal vector along $\gamma$ at the corresponding point. 
We parameterize $\gamma$ as $\boldsymbol{\gamma}(\theta)=({\gamma}_{1}(\theta),{\gamma}_{2}(\theta))$, where
$$
    \begin{cases}
        \begin{array}{cl}
            {\gamma}_{1}(\theta) = p_K(\theta)\cos\theta-p'_K(\theta)\sin\theta,\\
            {\gamma}_{2}(\theta) = p_K(\theta)\sin\theta+p'_K(\theta)\cos\theta.\\
        \end{array}
    \end{cases}
$$
Therefore, the curvature $\kappa(\theta)$ and curvature radius $\rho (\theta)$ of $\gamma$ are given by 
$$
\kappa(\theta) 
= \frac{\gamma_{1}'(\theta)\gamma_{2}''(\theta) - \gamma_{2}'(\theta)\gamma_{1}''(\theta)}
{(\gamma_{1}'(\theta)^2 + \gamma_{2}'(\theta)^2)^{3/2}}
 = \frac{1}{p_K(\theta)+p''_K(\theta)}
$$
and 
\begin{equation}\label{2.1}
\rho (\theta)
= \frac{1}{\kappa(\theta)} 
= p_K(\theta)+p''_K(\theta) .
\end{equation}

\begin{proposition}\label{P2.1}
    $K$ is a disk if and only if $p_K(\theta)$ is of the form 
    $p_K(\theta)=a_0 + a_1 \cos\theta + b_1 \sin\theta.$ 
\end{proposition}
\begin{proof}[proof.]
    If $K$ is a disk, then $\rho(\theta)$ is constant. Hence
    $$
        0 = \rho'(\theta) = p_K'(\theta) + p_K'''(\theta),
        \ \ \forall \ \theta \in [0, 2\pi].
    $$
    Since $p_K(\theta)$ is always continuous,  $2\pi$-periodic and bounded, it has Fourier expansion of the form
    \begin{equation}\label{2.2}
        p_K(\theta)=a_0+\sum_{n\geq 1}\left(a_n\cos n\theta+b_n\sin n\theta\right).
    \end{equation}
Then we have
$$
0 
= \int_0^{2\pi}\left(p'_K(\theta)+p'''_K(\theta)\right)^2\mathrm{d}\theta
= \pi\sum_{n\geq 2}\left(n^3 - n\right)^2\left(a_n^2 + b_n^2\right),
$$
which implies that $a_n=b_n=0$ for $n \geq 2$.

Conversely, if $p_K(\theta)=a_0 + a_1 \cos\theta + b_1 \sin\theta$, $\gamma$ is parameterized as 
$$
    \begin{cases}
        \begin{array}{cl}
            {\gamma}_{1}(\theta) = a_1 + a_0\cos\theta,\\
            {\gamma}_{2}(\theta) = b_1 + a_0\sin\theta.\\
        \end{array}
    \end{cases}
$$
Hence $K$ is a disk centering at $(a_1, b_1)$ with radius $|a_0|$.
\end{proof}

The length $L$ of $\boldsymbol{\gamma} (\theta)$ and the area $A$ enclosed by $\gamma$ can be calculated, respectively, by 
\begin{equation}\label{2.3}
L=\oint_{\gamma}\mathrm{d}s
=\int_0^{2\pi}\rho(\theta)\mathrm{d}\theta
=\int_0^{2\pi}\left(p_K(\theta)+p''_K(\theta)\right)\mathrm{d}\theta
\end{equation} 
and 
\begin{equation}\label{2.4}
A=\frac{1}{2}\oint_{\gamma}\gamma_1\mathrm{d}\gamma_2 - \gamma_2\mathrm{d}\gamma_1 = \frac12 \int_0^{2\pi}\left(p_K(\theta)^2-p'_K(\theta)^2\right)\mathrm{d}\theta.
\end{equation}
Let $\boldsymbol{N}(\theta) = (-\cos\theta, -\sin\theta)$ be the unit inward normal vector field along $\gamma$. The curvature centers locus $\beta$ is given by 
$$
\boldsymbol{\beta}(\theta)
=\boldsymbol{\gamma}(\theta) + \rho(\theta)\boldsymbol{N}(\theta)
=({\beta}_{1}(\theta),{\beta}_{2}(\theta)),
$$
where
$$
    \begin{cases}
        \begin{array}{cl}
            {\beta}_{1}(\theta) 
            = -p'_K(\theta)\sin\theta-p''_K(\theta)\cos\theta,\\
            {\beta}_{2}(\theta) 
            = p'_K(\theta)\cos\theta-p''_K(\theta)\sin\theta.\\
        \end{array}
    \end{cases}
$$
The oriented area of the domain enclosed by $\beta$ is given by 
\begin{equation}\label{2.5}
\tilde A=\frac{1}{2}\oint_{\beta}\beta_1\mathrm{d}\beta_2 - \beta_2\mathrm{d}\beta_1=\frac12 \int_0^{2\pi}\left(p'_K(\theta)^2-p''_K(\theta)^2\right)\mathrm{d}\theta.
\end{equation}
Its curvature $\tilde \kappa(\theta)$ and the curvature radius $\rho_{\beta} (\theta)$ are given by 
$$
\tilde \kappa(\theta)
 = \frac{{\beta}_{1}'(\theta){\beta}_{2}''(\theta) - {\beta}_{2}'(\theta){\beta}_{1}''(\theta)}
{({\beta}_{1}'(\theta)^2 + {\beta}_{2}'(\theta)^2)^{3/2}}
 = \frac{1}{p'_K(\theta)+p'''_K(\theta)}
$$
and 
\begin{equation}\label{2.6}
\rho_{\beta} (\theta)
= \left| \frac{1}{\tilde \kappa(\theta)} \right|
= \left| p'_K(\theta)+p'''_K(\theta) \right|.
\end{equation}

From \eqref{2.1}--\eqref{2.6} we obtain the following expressions in terms of the Fourier coefficients of $p_K(\theta)$.
\begin{proposition}\label{P1}
The geometric quantities appearing in \eqref{1.10} satisfy 
\begin{align*}
\int_0^{2\pi}\rho(\theta)^2\,\mathrm{d}\theta
    &= 2\pi a_0^2 + \pi\sum_{n\geq 2} \left(n^2-1\right)^2 \left(a_n^2 + b_n^2\right), \\
L^2 &= 4\pi^2 a_0^2, \\
A   &= \pi a_0^2 - \frac{\pi}{2}\sum_{n\geq 2} \left(n^2-1\right)\left(a_n^2 + b_n^2\right), \\
\left|\tilde A\right| &= \frac{\pi}{2}\sum_{n\geq 2} n^2\left(n^2-1\right)\left(a_n^2 + b_n^2\right), \\
\int_0^{2\pi}\rho_{\beta}(\theta)^2\,\mathrm{d}\theta
    &= \pi \sum_{n\geq 2} n^2\left(n^2-1\right)^2 \left(a_n^2 + b_n^2\right).
\end{align*}
\end{proposition}

We use Lemma \ref{L1} to estimate $\oint_{\gamma}\kappa^2\mathrm{d}s$.
\begin{lemma}\label{L1}
For the curvature $\kappa$ of $\gamma$, 
\begin{equation}\label{2.7}
    \oint_{\gamma}\kappa^2\mathrm{d}s \geq 3L - 4A - 4\left|\tilde A\right|,
\end{equation}
with equality only if $K$ is a disk.
\end{lemma}
\begin{proof}[proof.]
By Cauchy-Schwarz inequality, 
$$
4\pi^2 
\leq 
\int_0^{2\pi}\frac{\mathrm{d}\theta}{\kappa} \int_0^{2\pi}\kappa\mathrm{d}\theta
= \oint_{\gamma}\mathrm{d}s \oint_{\gamma}\kappa^2\mathrm{d}s
= L\oint_{\gamma}\kappa^2\mathrm{d}s.
$$
Using the reverse isoperimetric inequality \eqref{1.5}, we have
$$
A + \left|\tilde A\right|
\geq \frac{L^2}{4\pi}.
$$
Hence 
\begin{align*}
4A + 4\left|\tilde A\right| + \oint_{\gamma}\kappa^2\mathrm{d}s 
& = 2\left(A + \left|\tilde A\right|\right) 
+ 2\left(A + \left|\tilde A\right|\right)
+ \oint_{\gamma}\kappa^2\mathrm{d}s\\
& \geq 3 \cdot \sqrt[3]
{4\left(A + \left|\tilde A\right|\right)^2 
\cdot 
\oint_{\gamma}\kappa^2\mathrm{d}s} \\
& \geq 3 \cdot \sqrt[3]
{4 \cdot \left(\frac{L^2}{4\pi}\right)^2 \cdot \frac{4\pi^2}{L}}\\
& \geq 3 L.
\end{align*}
By the equality condition of \eqref{1.5}, the equality holds in \eqref{2.7} only if $K$ is a disk.
\end{proof}

\section{Proof and application of Theorem \ref{T1}}\label{S3}

\begin{proof}[Proof of Theorem \ref{T1}.]

\begin{enumerate}[(i)]
    \item 
    It follows from inequality \eqref{1.3} and \eqref{1.4} that when $\lambda \leq 0$ and $\xi \geq 0$, we have
$$
\lambda\left(\rho_e - \rho_i\right)^2 \geq \frac{\lambda }{\pi^2}L^2 - \frac{4\lambda}{\pi}A.
$$
and
$$
    \xi\left(\rho_M - \rho_m\right)^2 \geq \frac{\xi }{\pi^2}L^2 - \frac{4\xi}{\pi}A
$$
Together with Lemma \ref{L1} and $\zeta \geq 0$, we have
$$
\zeta\oint_{\gamma}\kappa^2\mathrm{d}s 
\geq \zeta\left(3L - 4A - 4\left|\tilde A\right|\right).
$$
Hence
\begin{align*}
W
& \geq \alpha \int_0^{2\pi}\rho(\theta)^2\mathrm{d}\theta
+ \left(\delta + \frac{\lambda + \xi}{\pi^2}\right) L^2 
+ 3\zeta L 
\\
& \quad + \left(\mu - \frac{4}{\pi}(\lambda + \xi) - 4\zeta\right) A  + \left(\sigma - 4\zeta\right) \left|\tilde A\right| + \eta \int_0^{2\pi} \rho_{\beta}(\theta)^2\mathrm{d}\theta.
\end{align*}
Then by inequality \eqref{1.9} and Proposition \ref{P1}, we have 
\begin{equation}\label{3.1}
    \begin{split}
W
& \geq \alpha\left(2\pi a_0^2 + \pi\sum_{n\geq 2}\left(n^2-1\right)^2\left(a_n^2 + b_n^2\right)\right)\\
&\quad + \left(\delta + \frac{\lambda + \xi}{\pi^2}\right)4\pi^2a_0^2 + 6\pi\zeta a_0\\
&\quad + \left(\mu - \frac{4}{\pi}(\lambda + \xi) - 4\zeta\right)\left( \pi a_0^2-\frac{\pi}{2}\sum_{n\geq 2}\left(n^2-1\right)\left(a_n^2+b_n^2\right) \right)\\
&\quad + \frac{\pi}{2}(\sigma - 4\zeta)\sum_{n\geq 2}n^2\left(n^2-1\right)\left(a_n^2+b_n^2\right)\\
& \quad + \pi \eta \sum_{n\geq 2}n^2\left(n^2-1\right)^2\left(a_n^2+b_n^2\right)\\
& = 6\pi\zeta a_0 
+ \pi\left(2\alpha + 4\pi\delta + \mu - 4\zeta\right)a_0^2\\ 
& \quad + \frac\pi2
\sum_{n\geq 2}\left(n^2-1\right)
\cdot
f(n)
\cdot
\left(a_n^2 + b_n^2\right)\\
& \geq \frac\pi2
\sum_{n\geq 2}\left(n^2-1\right)
\cdot
f(n)
\cdot
\left(a_n^2 + b_n^2\right),
\end{split}
\end{equation}

where 
$$
f(n)=2\eta n^4 
    + \left(
        2\alpha + \sigma - 2\eta - 4\zeta 
    \right)n^2 
    + \left(
        - 2\alpha - \mu + \frac{4}{\pi}\lambda + \frac{4}{\pi}\xi + 4\zeta
    \right). 
$$
It follows from \eqref{1.9} that the polynomial $f(n)$ is an increasing function for $n \geq 2$, hence
\begin{align*}
W & \geq \frac\pi2
\sum_{n\geq 2}\left(n^2-1\right)
\cdot
f(n)
\cdot
\left(a_n^2 + b_n^2\right)\\
& \geq \frac\pi2
\sum_{n\geq 2}3f(2)
\cdot
\left(a_n^2 + b_n^2\right) \\
& = \frac{3}{2}\sum_{n\geq 2}\left(6\pi\alpha - \pi\mu + 4\pi\sigma + 24\pi\eta + 4\lambda + 4\xi - 12\pi\zeta\right)\left(a_n^2 + b_n^2\right)\\
& \geq 0.
\end{align*}
Hence $W \geq 0$,
which proves inequality \eqref{1.10}.
    \item
    If $K$ is a disk, then $\rho_M = \rho_m = \rho_e = \rho_i$, $\left|\tilde A\right|=0$ and $\rho_{\beta}=0$. According to the equality conditions in \eqref{1.1} and \eqref{1.6}, we have
$$L^2 = 4\pi A$$
and
$$\int_0^{2\pi}\rho(\theta)^2\mathrm{d}\theta=\frac{L^2-2\pi A}{\pi}=2A.$$
Hence
\begin{align*}
W
&=(2\alpha + 4\pi \delta + \mu )A + \zeta\oint_{\gamma}\kappa^2\mathrm{d}s.
\end{align*}
Then for the parameter set satisfying \eqref{1.11}
we have
$$W = 0.$$
    \item
    Conversely, if equality holds in \eqref{1.10}, the first inequality in \eqref{3.1} changes to equality, implying that
$$
\lambda\left(\rho_e - \rho_i\right)^2
= \frac{\lambda}{\pi^2}L^2 - \frac{4\lambda}{\pi}A,
$$
$$
\xi\left(\rho_M - \rho_m\right)^2 
= \frac{\xi }{\pi^2}L^2 - \frac{4\xi}{\pi}A,
$$
and
$$
\zeta\oint_{\gamma}\kappa^2\mathrm{d}s 
= \zeta\left(3L - 4A - 4\left|\tilde A\right|\right).
$$
Combining \eqref{1.12} yields
$$
\left(\rho_e - \rho_i\right)^2 
= \frac{L^2 - 4\pi A}{\pi^2},
$$
$$
\left(\rho_M - \rho_m\right)^2 
= \frac{L^2 - 4\pi A}{\pi^2},
$$
or
$$
\oint_{\gamma}\kappa^2\mathrm{d}s 
= 3L - 4A - 4\left|\tilde A\right|.
$$
By the equality conditions of \eqref{1.3}, \eqref{1.4} and \eqref{2.7}, $K$ is a disk. 
This proves Theorem \ref{T1}.
\end{enumerate}
\end{proof}

\begin{remark}
    If the equality holds in \eqref{1.11} and the parameter set satisfies
    $$
    \begin{cases}
        \begin{array}{cl}
        \eta, \xi, \zeta \geq 0,\\
        \lambda \leq 0,\\
        2\alpha + 4\pi\delta + \mu - 4\zeta \geq 0,\\
        2\alpha + \sigma - 2\eta - 4\zeta \geq 0,\\
        6\pi\alpha - \pi\mu + 4\pi\sigma + 24\pi\eta + 4\lambda + 4\xi - 12\pi\zeta > 0,\\
        \end{array}
    \end{cases}
    $$
then $a_n=b_n=0$ for all $n \geq 2$. Hence the support function is of the form
$p_K(\theta)
= a_0 
+ a_1 \cos\theta + b_1 \sin\theta.$
By Proposition \ref{P2.1}, $K$ is a disk centering at $(a_1, b_1)$ with radius $|a_0|$.
\end{remark}

\begin{remark}
    If the equality holds in \eqref{1.11} and the parameter set satisfies
    $$
    \begin{cases}
        \begin{array}{cl}
        \eta > 0,\\
        \xi, \zeta \geq 0,\\
        \lambda \leq 0,\\
        2\alpha + 4\pi\delta + \mu - 4\zeta \geq 0,\\
        2\alpha + \sigma - 2\eta - 4\zeta \geq 0,\\
        6\pi\alpha - \pi\mu + 4\pi\sigma + 24\pi\eta + 4\lambda + 4\xi - 12\pi\zeta = 0,\\
        \end{array}
    \end{cases}
    $$
then $a_n=b_n=0$ for all $n > 2$.
Hence the support function is of the form
$p_K(\theta)
= a_0 
+ a_1 \cos\theta + b_1 \sin\theta
+ a_2 \cos2\theta + b_2 \sin2\theta.$ 
\end{remark}

\begin{remark}
The parametric isoperimetric-type inequality \eqref{1.10} is an improved version of proven results.
\begin{itemize}
    \item When 
    $(\alpha, \delta, \mu, \sigma, \eta, \lambda, \xi, \zeta)
    = (0, 1, -4\pi, 0, 0, 0, 0, 0)$, the inequality \eqref{1.10} turns into \eqref{1.1}.
    \item When  
    $(\alpha, \delta, \mu, \sigma, \eta, \lambda, \xi, \zeta)
    = (0, -1, 4\pi, 4\pi, 0, 0, 0, 0)$, \eqref{1.10} turns into \eqref{1.5}.
    \item When 
    $(\alpha, \delta, \mu, \sigma, \eta, \lambda, \xi, \zeta)
    = (\pi, -1, 2\pi, 0, 0, 0, 0, 0)$, \eqref{1.10} turns into \eqref{1.6}.
    \item If we select 
    $(\alpha, \delta, \mu, \sigma, \eta, \lambda, \xi, \zeta)
    = (0, -1, 4\pi, \pi, 0, 0, 0, 0)$, \eqref{1.10} turns into \eqref{1.7}.
    \item If we select 
    $(\alpha, \delta, \mu, \sigma, \eta, \lambda, \xi, \zeta)
    = (\pi, -1, 2\pi, -\pi, 0, 0, 0, 0)$, \eqref{1.10} turns into \eqref{1.8}.
\end{itemize}

Hence \eqref{1.10} could also be regarded as a family of reverse isoperimetric-type inequality. 
\end{remark}

\begin{corollary}\label{C1}
    If the parameter set satisfies \eqref{1.9}, \eqref{1.11} and \eqref{1.12}, i.e.
    \begin{equation}\label{3.2}
            \begin{cases}
                \begin{array}{cl}
            \eta, \xi \geq 0,\\
            \lambda \leq 0,\\
            \lambda^2 + \xi^2 > 0,\\
            \zeta = 0,\\
            2\alpha + 4\pi\delta + \mu = 0,\\
            2\alpha + \sigma - 2\eta \geq 0,\\
            6\pi\alpha - \pi\mu + 4\pi\sigma + 24\pi\eta + 4\lambda + 4\xi \geq   0,\\
                \end{array}
            \end{cases}
        \end{equation} then $K$ is a disk if and only if the equality holds in \eqref{1.10}. 
\end{corollary}

If we select other values of parameter set to satisfy \eqref{1.9}, we can derive some new inequalities: 
\begin{corollary}\label{C}
Let $K \subset \mathbb{R}^2$ be a closed and strictly convex domain, with geometric quantities $\rho$, $L$, $A$, $\tilde A$, $\rho_{\beta}$, $\rho_e$, $\rho_i$, $\rho_M$, $\rho_m$ and $\kappa$ defined as above. 
\begin{enumerate}[(i)]
    \item We have the following inequalities
\begin{equation}\label{3.3}
2L^2
+ 8\left|\tilde A\right|
+ 2\int_0^{2\pi}\rho_{\beta}(\theta)^2\mathrm{d}\theta
+ \oint_{\gamma}\kappa^2\mathrm{d}s
\geq
20A, 
\end{equation}
\begin{equation}\label{3.4}
\begin{split}
3\int_0^{2\pi}\rho(\theta)^2\mathrm{d}\theta
+ \left(32\pi - 38\right)A
+ (\pi - 1)\int_0^{2\pi}\rho_{\beta}(\theta)^2\mathrm{d}\theta\\
\geq
\frac{8\pi - 8}{\pi}L^2
+ ( 8 - 2\pi )\left|\tilde A\right|,
\end{split}
\end{equation}

\begin{equation}\label{3.5}
\int_0^{2\pi}\rho(\theta)^2\mathrm{d}\theta
+ \left(\frac{52\pi}{7} - 2\right)A
+ \frac{3}{4}\int_0^{2\pi}\rho_{\beta}(\theta)^2\mathrm{d}\theta
\geq
\frac{13}{7}L^2
+ \frac{\pi}{7}\left|\tilde A\right|
\end{equation}
and
\begin{equation}\label{3.6}
\left|\tilde A\right|
+ \epsilon\int_0^{2\pi}\rho_{\beta}(\theta)^2\mathrm{d}\theta
+ \pi \left(\rho_M - \rho_m\right)^2
\geq
2\pi \left(\rho_e - \rho_i\right)^2,
\end{equation}
where $\epsilon \in [0, \frac{1}{2}]$.

\item The inequality in \eqref{3.3} is strict, i.e.
$$
2L^2
+ 8\left|\tilde A\right|
+ 2\int_0^{2\pi}\rho_{\beta}(\theta)^2\mathrm{d}\theta
+ \oint_{\gamma}\kappa^2\mathrm{d}s
>
20A. 
$$

\item The equalities hold in \eqref{3.4} and \eqref{3.5} if $K$ is a disk.

\item The equality holds in \eqref{3.6} if and only if $K$ is a disk.
\end{enumerate}
\end{corollary}

\begin{proof}[proof.]
\begin{enumerate}[(i)]
    \item The parameter set 
    \begin{equation}\label{3.7}
    (\alpha, \delta, \mu, \sigma, \eta, \lambda, \xi, \zeta) 
    = (0, 2, -20, 8, 2, 0, 0, 1)        
    \end{equation}
    satisfies conditions \eqref{1.9}, then it follows from Theorem \ref{T1} (\ref{1.1.1}) that \eqref{3.3} is proved. Also, we can select 
    \begin{equation}\label{3.8}
    (\alpha, \delta, \mu, \sigma, \eta, \lambda, \xi, \zeta) 
    = \left(3, \frac{8-8\pi}{\pi}, 32\pi-38, 2\pi-8, \pi-1, 0, 0, 0\right),
    \end{equation}
    \begin{equation}\label{3.9}
    (\alpha, \delta, \mu, \sigma, \eta, \lambda, \xi, \zeta) 
    = \left(1, -\frac{13}{7}, \frac{52}{7}\pi-2, -\frac{\pi}{7}, \frac{3}{4}, 0, 0, 0\right)
    \end{equation} 
    and
    \begin{equation}\label{3.10}
    (\alpha, \delta, \mu, \sigma, \eta, \lambda, \xi, \zeta) 
    = (0, 0, 0, 1, \epsilon, -2\pi, \pi, 0),\ \ 
    \epsilon \in \left[0, \frac{1}{2}\right]
    \end{equation} 
    satisfying \eqref{1.9} to prove \eqref{3.4}-\eqref{3.6}.
    
    \item 
    Assume the equality holds in \eqref{3.3}. Parameter set \eqref{3.7} satisfies \eqref{1.12}, it follows from Theorem \ref{T1} (\ref{1.1.3}) that $K$ is a disk. Then by proof of Theorem \ref{T1} (\ref{1.1.2}), parameter set \eqref{3.7} satisfies \eqref{1.11}, which is contradiction. Hence the inequality in \eqref{3.3} is strict.
    
    \item 
    We can verify parameter sets \eqref{3.8} and \eqref{3.9} both satisfy \eqref{1.11}. It follows from Theorem \ref{T1} (\ref{1.1.2}) that the equalities hold in \eqref{3.4} and \eqref{3.5} if $K$ is a disk.
    
    \item 
    The parameter set \eqref{3.10} satisfies \eqref{3.2}. It follows form Corollary \ref{C1} that equality holds in \eqref{3.6} if and only if $K$ is a disk.
\end{enumerate}

\end{proof}

\begin{remark}
Inequalities \eqref{3.4} and \eqref{3.5} are improved versions of 
\begin{equation}\label{3.11}
    \int_0^{2\pi}\rho(\theta)^2\mathrm{d}\theta
    + 30A
    + \int_0^{2\pi}\rho_{\beta}(\theta)^2\mathrm{d}\theta
    \geq
    2L^2
\end{equation}
in \cite[Theorem 1.7]{LG}.    
Rewriting \eqref{3.4} gives
\begin{align*}
\int_0^{2\pi}\rho(\theta)^2\mathrm{d}\theta
+ (8\pi-2) A
+ \frac{\pi-1}{3}\int_0^{2\pi}\rho_{\beta}(\theta)^2\mathrm{d}\theta\\
\geq
\frac{8\pi-8}{3\pi}L^2
+ \frac{8-2\pi}{3\pi}
\left( 4\pi A + \pi\left|\tilde A\right| \right).
\end{align*}   
It follows from \eqref{1.7} that 
$$\frac{8\pi-8}{3\pi}L^2
+ \frac{8-2\pi}{3\pi}
\left( 4\pi A + \pi\left|\tilde A\right| \right)
\geq
\frac{8\pi-8}{3\pi}L^2 + \frac{8-2\pi}{3\pi}L^2
= 2L^2.
$$
Together with 
$$\frac{\pi-1}{3}\int_0^{2\pi}\rho_{\beta}(\theta)^2\mathrm{d}\theta
\leq
\int_0^{2\pi}\rho_{\beta}(\theta)^2\mathrm{d}\theta,
$$
and
$$
(8\pi-2) A
\leq
30A,
$$
\eqref{3.4} implies \eqref{3.11}, and hence is an improved version of \eqref{3.11}. 
A similar argument applies to \eqref{3.5}.
\end{remark}

\section{Stability of isoperimetric-type inequalities}\label{S4}

\begin{definition}\label{D}
    Let $K, K' \subset \mathbb{R}^2$ be two convex domains with support functions $p_K$, $p_{K'}$, respectively. 
    \begin{enumerate}[(i)]
        \item The $h_1$ norm between $K$ and $K'$ is
        \begin{equation}\label{4.1}
        h_1(K, K')=\max_{\theta \in [0, 2\pi]}\left|p_K(\theta)-p_{K'}(\theta)\right|.
        \end{equation}
    
        \item The $h_2$ norm between $K$ and $K'$ is 
        \begin{equation}\label{4.2}
        h_2(K, K')=\left(\int_0^{2\pi}\left|p_K(\theta)-p_{K'}(\theta)\right|^2\mathrm{d}\theta\right)^{\frac{1}{2}}.
        \end{equation}
    \end{enumerate}
\end{definition}

\begin{proposition}
    Both norms satisfy $h_k(K, K')=0$ if and only if $K = K'$ for $k = 1, 2$.
\end{proposition}

\begin{definition}\label{D1}
The Steiner point $\boldsymbol{C}(K)$ of convex domain $K$ is
\begin{equation}\label{4.3}
\boldsymbol{C}(K)
= \frac{1}{\pi}
\int_0^{2\pi}p_K(\theta)(\cos\theta, \sin\theta)\mathrm{d}\theta.
\end{equation}
The Steiner disk $S(K)$ of $K$ centers at $\boldsymbol{C}(K)$ with radius $\frac{L(K)}{2\pi}$, where $L(K)$ is the perimeter of $K$.
\end{definition}

The following proofs give the stability of inequality \eqref{1.10} with respect to the $h_1$ and $h_2$ norms.

\begin{proof}[Proof of Theorem \ref{T2}.]
\begin{enumerate}[(i)]
    \item 
    Assume $\boldsymbol{C}(K)=(0,0)$, then it follows from \eqref{2.2} and \eqref{4.3} that $a_1 = b_1 =0$. Since $L(K)=2\pi a_0$, the support functions $p_K$ and $p_{S(K)}$ have Fourier expansions
\begin{equation}\label{4.4}
p_K(\theta)=\frac{L(K)}{2\pi}+\sum_{n\geq 2}\left(a_n\cos n\theta+b_n\sin n\theta\right)
\end{equation}
and
\begin{equation}\label{4.5}
p_{S(K)}(\theta)=\frac{L(K)}{2\pi}.
\end{equation}
From the proof of Theorem \ref{T1}, we obtain
\begin{equation}\label{4.6}
6\pi\zeta a_0
+ \pi\left(2\alpha + 4\pi \delta + \mu - 4\zeta\right)a_0^2 
+ \frac\pi2
\sum_{n\geq 2}\left(n^2-1\right)
\cdot
f(n)
\cdot
\left(a_n^2 + b_n^2\right)
\leq 
W.
\end{equation} 
Using \eqref{4.4} and \eqref{4.5}, we have
\begin{align*}
\left|p_K(\theta)-p_{S(K)}(\theta)\right|
&=\left|\frac{L(K)}{2\pi}+\sum_{n\geq 2}\left(a_n\cos n\theta+b_n\sin n\theta\right)-\frac{L(K)}{2\pi}\right|\\
&\leq\sum_{n\geq 2}\left|a_n \cos n\theta+b_n\sin n\theta\right|\\
&\leq\sum_{n\geq 2}\sqrt{a_n^2+b_n^2}
\end{align*}
for all $\theta \in [0,2\pi]$. By \eqref{4.1} and Cauchy-Schwarz Inequality, 
\begin{equation}\label{4.7}
\begin{split}
& \quad h_1\left(K, S(K)\right)^2\\
&\leq \left(\sum_{n\geq 2}\sqrt{a_n^2+b_n^2}\right)^2\\
&\leq \left(\frac\pi2
\sum_{n\geq 2}
\left(n^2-1\right)
\cdot f(n) \cdot \left(a_n^2+b_n^2\right)\right)
 \left(\frac{2}{\pi}\sum_{n\geq 2}
\frac{1}
{\left(n^2-1\right)
\cdot f(n)  }\right)\\
&\leq 
\left(
6\pi\zeta a_0
+ \pi\left(2\alpha + 4\pi\delta + \mu - 4\zeta\right)a_0^2
\right)\cdot 1 \\
& \quad + \left(\frac\pi2
\sum_{n\geq 2}
\left(n^2-1\right)
\cdot f(n) \cdot \left(a_n^2+b_n^2\right)\right)
 \left(\frac{2}{\pi}\sum_{n\geq 2}
\frac{1}
{\left(n^2-1\right)
\cdot f(n)}\right).
\end{split}
\end{equation} 
For all $n \geq 2$, we have
$$
f(n)\geq f(2)
= \frac{1}{\pi} (6\pi\alpha - \pi\mu + 4\pi\sigma + 24\pi\eta + 4\lambda + 4\xi - 12\pi\zeta) > 0.
$$
Hence 
\begin{align*}
\quad \frac{2}{\pi}\sum_{n\geq 2}
\frac{1}
{\left(n^2-1\right)
\cdot f(n) } 
&\leq 
\frac{2}{\pi \cdot f(2)}
\sum_{n\geq 2}\frac{1}{n^2-1}\\
& = \frac{3}{2(6\pi\alpha - \pi\mu + 4\pi\sigma + 24\pi\eta + 4\lambda + 4\xi - 12\pi\zeta)}.
\end{align*}
Combining \eqref{4.6} and \eqref{4.7} yields
$$
h_1\left(K, S(K)\right)^2
\leq
C(\alpha, \mu, \sigma, \eta, \lambda, \xi, \zeta)W
$$

where
\[  C(\alpha, \mu, \sigma, \eta, \lambda, \xi, \zeta)=\max \left\{
1,
\frac{3}{2(6\pi\alpha - \pi\mu + 4\pi\sigma + 24\pi\eta + 4\lambda + 4\xi - 12\pi\zeta)}
\right\}. \]
This proves inequality \eqref{1.14}.
    \item
    If $K$ is a disk, the proof of Theorem \ref{T1} gives
$$
W = \left(2\alpha + 4\pi\delta + \mu\right)A 
+ \zeta \oint_{\gamma}\kappa^2\mathrm{d}s.
$$
If the parameter set satisfies \eqref{1.11},
we have $ W=0$. 
$K$ is a disk, implying that $h_1(K, S(K))=0$, hence the equality holds in \eqref{1.14}. 
This proves Theorem \ref{T2}.
\end{enumerate}
\end{proof}

\begin{proof}[Proof of Theorem \ref{T3}.]
\begin{enumerate}[(i)]
    \item 
    Assume $\boldsymbol{C}(K)=(0,0)$.
    By \eqref{4.2} and Parseval's equality,
\begin{align*}
& \quad h_2\left(K, S(K)\right)^2\\
& = \int_0^{2\pi}\left|p_K(\theta)-p_{S(K)}(\theta)\right|^2\mathrm{d}\theta\\
& = \pi\sum_{n\geq 2}\left(a_n^2+b_n^2\right)\\
& = C(\alpha, \mu, \sigma, \eta, \lambda, \xi, \zeta) 
\cdot 
\frac{\pi}{2}\sum_{n\geq 2} 
3f(2)
\cdot
\left(a_n^2+b_n^2\right).
\end{align*}
Using \eqref{1.13} and \eqref{3.1}, we have
\begin{align*}
h_2\left(K, S(K)\right)^2
& \leq C(\alpha, \mu, \sigma, \eta, \lambda, \xi, \zeta) 
\cdot 
\frac{\pi}{2}
\sum_{n\geq 2}
\left(n^2-1\right)
\cdot f(n) \cdot
\left(a_n^2+b_n^2\right)\\
& \leq C(\alpha, \mu, \sigma, \eta, \lambda, \xi, \zeta) W.
\end{align*}
This proves inequality \eqref{1.15}.

    \item
    If $K$ is a disk, the proof of Theorem \ref{T1} gives
$$
W = \left(2\alpha + 4\pi\delta + \mu\right)A 
+ \zeta \oint_{\gamma}\kappa^2\mathrm{d}s.
$$
If the parameter set satisfies \eqref{1.11},  
we have $ W=0$. $K$ is a disk, implying that $h_2(K, S(K))=0$, hence the equality holds in \eqref{1.15}. 

    \item
    Conversely, if equality holds in \eqref{1.15}, all inequalities in 
\begin{align*}
W 
& \geq \alpha \int_0^{2\pi}\rho(\theta)^2\mathrm{d}\theta
+ \left(\delta + \frac{\lambda + \xi}{\pi^2}\right) L^2 
+ 3\zeta L
+ \left(\mu - \frac{4}{\pi}(\lambda + \xi) - 4\zeta\right) A \\
& \quad + \left(\sigma - 4\zeta\right) \left|\tilde A\right| + \eta \int_0^{2\pi} \rho_{\beta}(\theta)^2\mathrm{d}\theta\\
& = 6\pi\zeta a_0 
+ \pi\left(2\alpha + 4\pi\delta + \mu - 4\zeta\right)a_0^2 + \frac\pi2
\sum_{n\geq 2}\left(n^2-1\right)
\cdot f(n) \cdot
\left(a_n^2 + b_n^2\right)\\
& \geq \frac{\pi}{2}
\sum_{n\geq 2}\left(n^2-1\right)
\cdot f(n) \cdot
\left(a_n^2 + b_n^2\right)\\
& \geq \frac\pi2 \sum_{n\geq 2}
3\left( 6\alpha - \mu + 4\sigma + 24\eta + \frac{4}{\pi}\lambda + \frac{4}{\pi}\xi - 12\zeta \right)\left(a_n^2 + b_n^2\right)\\
& = \frac{h_2\left(K, S(K)\right)^2}{C(\alpha, \mu, \sigma, \eta, \lambda, \xi, \zeta)} 
\end{align*}
become equalities.
As proof of Theorem \ref{T1}, equality holds in \eqref{1.3}, \eqref{1.4} and \eqref{2.7}, which implies that $K$ is a disk. 
This proves Theorem \ref{T3}.
\end{enumerate}
\end{proof}

\begin{definition}
    Inequality \eqref{1.10} is called $h_k$-stable, if there exists a positive constant $C$ depending only on parameter set, such that
    $$
    h_k(K,S(K))^2 \leq CW,\ \ k=1, 2.
    $$
    \eqref{1.10} is called stable, if it is both $h_1$-stable and $h_2$-stable;
    otherwise it is unstable.
\end{definition}

\begin{corollary}\label{CC}
    For parameter set satisfying \eqref{1.13}, inequality \eqref{1.10} is stable.
\end{corollary}
\begin{proof}[proof.]
     It follows from Theorem \ref{T2} (\ref{1.2.1}) and Theorem \ref{T3} (\ref{1.3.1}) that \eqref{1.10} is both $h_1$-stable and $h_2$-stable, hence \eqref{1.10} is stable.
\end{proof}

\begin{remark}
Applying the method of \cite[Theorem 4.3]{G2} and \cite[Theorem 4.3]{LG}, we obtain stability estimate 
$$
h_2(K, S(K))^2 \leq W
$$
in Theorem \ref{T3}. But this requires
$$
3\left(6\alpha - \mu + 4\sigma + 24\eta + \frac{4}{\pi}\lambda + \frac{4}{\pi}\xi - 12\zeta\right) - 2 \geq 0,
$$
which is more restrictive than 
\begin{equation}\label{4.8}
6\pi\alpha - \pi\mu + 4\pi\sigma + 24\pi\eta + 4\lambda + 4\xi - 12\pi\zeta > 0
\end{equation}
in \eqref{1.13}.
Hence condition \eqref{4.8} represents a weaker condition for stability. 
\end{remark}

\begin{remark}\label{R}
When $\zeta = \xi = \lambda =0$, the condition \eqref{4.8} in \eqref{1.13} cannot be relaxed as in \eqref{1.9}, i.e.
$$6\pi\alpha - \pi\mu + 4\pi\sigma + 24\pi\eta + 4\lambda + 4\xi - 12\pi\zeta \geq 0.$$
If
$$
f(2)
= \frac{1}{\pi} (6\pi\alpha - \pi\mu + 4\pi\sigma + 24\pi\eta + 4\lambda + 4\xi - 12\pi\zeta) = 0,
$$
then 
$$
W = \pi\left(2\alpha + 4\pi\delta + \mu \right)a_0^2 
+ \frac\pi2
\sum_{n\geq 3}\left(n^2-1\right)
\cdot
f(n)
\cdot
\left(a_n^2 + b_n^2\right).
$$
Together with
$$
h_2\left(K, S(K)\right)^2 
= \pi\sum_{n\geq 2}\left(a_n^2+b_n^2\right)
\leq CW,
$$
any admissible constant $C$ would depend on $a_2$, $b_2$ and cannot be universal. Hence the inequality \eqref{1.10} is unstable.

\end{remark}

\begin{example}
Consider the inequalities \eqref{3.4}-\eqref{3.6} from Corollary \ref{C}.
    \begin{enumerate}[(i)]
        \item Inequality \eqref{3.4} is unstable, whereas \eqref{3.5} is stable.
        
        \item\label{2)}  \eqref{3.6} is stable for $\epsilon \in (0, \frac{1}{2}]$.
  
    \end{enumerate}
\end{example}
\begin{proof}[proof.]
\begin{enumerate}[(i)]
    \item The parameter set \eqref{3.8} satisfies $\lambda = \xi = \zeta =0$ and
        $$
            6\pi\alpha - \pi\mu + 4\pi\sigma + 24\pi\eta + 4\lambda + 4\xi - 12\pi\zeta = 0,
        $$
        By Remark \ref{R}, \eqref{3.4} is unstable. 
        In contrast, \eqref{3.9} satisfies \eqref{4.8}, because
        $$
            6\pi\alpha - \pi\mu + 4\pi\sigma + 24\pi\eta + 4\lambda + 4\xi - 12\pi\zeta
            = 26\pi - 8\pi^2 
            >0,
        $$
        which implies that \eqref{3.5} is stable.
    \item For $\epsilon \in (0, \frac{1}{2}]$, the parameter set \eqref{3.10} satisfies \eqref{4.8} and therefore \eqref{1.13}. It follows from Corollary \ref{CC} that \eqref{3.6} is stable.
    
\end{enumerate}
    
\end{proof}

The conditions $\eta, \xi, \zeta \geq 0$ and $\lambda \leq 0$ are essential in proof of Theorem \ref{T1}, giving inequalities of the form
$$
\hat{\eta} \int_0^{2\pi}\rho_{\beta}(\theta)^2\mathrm{d}\theta
+ \hat{\xi} (\rho_{M}-\rho_{m})^2
+ \hat{\zeta} \oint_{\gamma}\kappa^2\mathrm{d}s
+ B_1
\geq
 B_2 + \hat{\lambda}(\rho_e - \rho_i)^2,
$$
where $\hat{\eta}, \hat{\zeta}, \hat{\xi}, \hat{\lambda}$ are non-negative parameters, and $B_1, B_2$ are sums of non-negative geometric quantities.
We wonder whether introducing new geometric quantities can yield inequalities of the reverse form
$$
\hat{\lambda}(\rho_e - \rho_i)^2 + B_1
\geq
\hat{\eta} \int_0^{2\pi}\rho_{\beta}(\theta)^2\mathrm{d}\theta
+ \hat{\xi} (\rho_{M}-\rho_{m})^2
+ \hat{\zeta} \oint_{\gamma}\kappa^2\mathrm{d}s
+ B_2,
$$
and whether such inequalities have stability.


\end{document}